\documentclass[a4paper,english,fontsize=11pt,parskip=half,abstracton]{scrartcl}
\usepackage{babel}
\usepackage[utf8]{inputenc}
\usepackage[T1]{fontenc}
\usepackage[a4paper,left=20mm,right=20mm,top=30mm,bottom=30mm]{geometry}
\usepackage{amsmath}
\usepackage{amsthm}
\usepackage{amssymb}
\usepackage{enumerate}
\usepackage{aliascnt}
\usepackage{tikz}
\usepackage[bookmarks=true,
            pdftitle={Orders generated by character values},
            pdfauthor={Benjamin Sambale},
            pdfkeywords={},
            pdfstartview={FitH}]{hyperref}

\newtheorem*{ThmA}{Theorem~A}
\newtheorem*{ThmB}{Theorem~B}
\newtheorem*{ConC}{Conjecture~C}
\newtheorem{Thm}{Theorem} 
\newaliascnt{Lem}{Thm}
\newtheorem{Lem}[Lem]{Lemma}
\aliascntresetthe{Lem}
\newaliascnt{Prop}{Thm}
\newtheorem{Prop}[Prop]{Proposition}
\aliascntresetthe{Prop}
\newaliascnt{Cor}{Thm}
\newtheorem{Cor}[Cor]{Corollary}
\aliascntresetthe{Cor}

\numberwithin{equation}{section}

\allowdisplaybreaks[1]

\renewcommand{\phi}{\varphi}
\newcommand{\C}{\mathrm{C}}
\newcommand{\N}{\mathrm{N}}
\newcommand{\Z}{\mathrm{Z}}
\newcommand{\ZZ}{\mathbb{Z}}
\newcommand{\CC}{\mathbb{C}}
\newcommand{\QQ}{\mathbb{Q}}

\newcommand{\NN}{\mathbb{N}}

\newcommand{\SL}{\operatorname{SL}}

\newcommand{\PSL}{\operatorname{PSL}}

\newcommand{\Irr}{\operatorname{Irr}}
\newcommand{\Gal}{\operatorname{Gal}}

\newcommand{\lcm}{\operatorname{lcm}}
\newcommand{\Tr}{\operatorname{Tr}}

\mathchardef\ordinarycolon\mathcode`\:  
 \mathcode`\:=\string"8000
 \begingroup \catcode`\:=\active
   \gdef:{\mathrel{\mathop\ordinarycolon}}
 \endgroup

\title{Orders generated by character values}
\author{Andreas Bächle\footnote{Vakgroep Wiskunde, Vrije Universiteit Brussel, Pleinlaan 2, 1050 Brussels, Belgium, \href{mailto:andreas.bachle@vub.be}{andreas.bachle@vub.be}} \ and Benjamin Sambale\footnote{Institut für Mathematik, Friedrich-Schiller-Universität Jena, 07737 Jena, Germany, 
\href{mailto:benjamin.sambale@uni-jena.de}{benjamin.sambale@uni-jena.de}}}
\date{\today}

\begin{document}
\frenchspacing
\maketitle
\begin{abstract}\noindent
Let $K:=\QQ(G)$ be the number field generated by the complex character values of a finite group $G$. Let $\ZZ_K$ be the ring of integers of $K$. In this paper we investigate the suborder $\ZZ[G]$ of $\ZZ_K$ generated by the character values of $G$. We prove that every prime divisor of the order of the finite abelian group $\ZZ_K/\ZZ[G]$ divides $|G|$. Moreover, if $G$ is nilpotent, we show that the exponent of $\ZZ_K/\ZZ[G]$ is a proper divisor of $|G|$ unless $G=1$. We conjecture that this holds for arbitrary finite groups $G$.
\end{abstract}


\section{Introduction}
It is well-known that the complex character values of a finite group $G$ are algebraic integers. We like to measure how “many” algebraic integers actually arise in this way. 
The field
\[K:=\QQ(G):=\QQ(\chi(g):\chi\in\Irr(G),\ g\in G)\subseteq\CC\] 
of character values of $G$ is contained in $\QQ_{\exp(G)}$ where $\exp(G)$ denotes the exponent of $G$ and $\QQ_n$ is the cyclotomic field generated by the complex $n$-th roots of unity. 
Let $\ZZ_K$ be the ring of integers of $K$. 
The character values of $G$ also generate an order $\ZZ[G]$ contained in $\ZZ_K$ (here $\ZZ[G]$ is neither the group algebra nor the ring of generalized characters). The deviation of $\ZZ[G]$ from $\ZZ_K$ can be measured by the structure of the finite abelian group $\ZZ_K/\ZZ[G]$. If $G$ is a rational group for instance, then $K=\QQ$ and $\ZZ[G]=\ZZ=\ZZ_K$. If $G$ is abelian, then $K=\QQ_{\exp(G)}$  and $\ZZ_K=\ZZ[e^{2\pi \sqrt{-1}/\exp(G)}]$. In this case it is easy to see that $\ZZ[G]=\ZZ_K$ as well.
On the other hand, we construct a group $G$ of order $240$ such that
\[ \ZZ_K/\ZZ[G]\cong C_{120}^2\times C_{60}^2\times C_{12}^4\times C_4^4\times C_2^{14}\]
where $C_n$ denotes a cyclic group of order $n$.
Nevertheless, our main theorems show that the structure of $\ZZ_K/\ZZ[G]$ is restricted by the order of $G$.

\begin{ThmA}
Let $G$ be a finite group and $K:=\QQ(G)$. Then the prime divisors of $|\ZZ_K/\ZZ[G]|$ divide $|G|$. 
\end{ThmA}

\begin{ThmB}
Let $G\ne 1$ be a nilpotent group and $K:=\QQ(G)$. Then the exponent of $\ZZ_K/\ZZ[G]$ is a proper divisor of $|G|$. In particular, $|G|\ZZ_K\subseteq\ZZ[G]$.
\end{ThmB}

In the final section we exhibit many examples which indicate that Theorem~B might be true without the nilpotency hypothesis.

\begin{ConC}\label{conj}
Let $G\ne 1$ be a finite group and $K:=\QQ(G)$. Then the exponent of $\ZZ_K/\ZZ[G]$ is a proper divisor of $|G|$.
\end{ConC}

\section{Preliminaries}
In addition to the notation introduced above, we define 
\begin{align*}
\QQ(g)&:=\QQ(\chi(g):\chi\in\Irr(G))&&(g\in G),\\
\ZZ[g]&:=\ZZ[\chi(g):\chi\in\Irr(G)],\\
\QQ(\chi)&:=\QQ(\chi(g):g\in G)&&(\chi\in\Irr(G)),\\
\ZZ[\chi]&:=\ZZ[\chi(g):g\in G].
\end{align*}
For number fields $K\subseteq L$ we denote the relative discriminant of $L$ with respect to $K$ by $d_{L|K}\in\ZZ_K$. If $K=\QQ$ we write $d_L:=d_{L|\QQ}$ as usual. 
We make use of the following tools from algebraic number theory.

\begin{Prop}\label{disc1}
The discriminant of any subfield of $\QQ_n$ divides $n^{\phi(n)}$. 
\end{Prop}
\begin{proof}
If $n=p^m$ is a power of a prime $p$, then by \cite[Lemma~I.10.1]{NeukirchE} the discriminant $d_n$ of $\QQ_n$ is $\pm p^{p^{m-1}(mp-m-1)}$, a divisor of $n^{\phi(n)}=p^{mp^{m-1}(p-1)}$. 
For arbitrary $n$ we obtain $d_n\mid n^n$ from \cite[Proposition~I.2.11]{NeukirchE}. Now if $K\subseteq\QQ_n$ is any subfield, then by \cite[Corollary~III.2.10]{NeukirchE} even $d_K^{|\QQ_n:K|}$ divides $d_n$.
\end{proof}

Although we only need a weak version of the following result, it seems worth stating a strong form.

\begin{Prop}\label{disc2}
Let $K$ and $L$ be Galois number fields. Then 
\[\gcd(d_K,d_L)\ZZ_{KL}\subseteq\frac{\gcd(d_K,d_L)}{d_{K\cap L}^m}\ZZ_{KL}\subseteq\ZZ_K\ZZ_L\] 
where $m:=\min\{|KL:K|,|KL:L|\}$.
In particular, $\ZZ_{KL}=\ZZ_K\ZZ_L$ if $d_K$ and $d_L$ are coprime.
\end{Prop}
\begin{proof}
Most textbooks only deal with the last claim. 
To prove the general case we follow \cite[Proposition~I.2.11]{NeukirchE}: 

We consider the compositum $KL$ as an extension over $M:=K\cap L$. Note that $\ZZ_{KL}$ ($\ZZ_K$, $\ZZ_L$ respectively) is the integral closure of $\ZZ_M$ in $KL$ ($K$, $L$ respectively). Let $b_1,\ldots,b_n$ be a $\ZZ_M$-basis of $\ZZ_K$ and let $c_1,\ldots,c_m$ be a $\ZZ_M$-basis of $\ZZ_L$. Then $\{b_ic_j:i=1,\ldots,n,j=1,\ldots,m\}$ is an $M$-basis of $KL$ as is well-known. Let $\alpha\in\ZZ_{KL}$ be arbitrary and write
\[\alpha=\sum_{i,j}a_{ij}b_ic_j\]
with $a_{ij}\in M$ for all $i,j$. Since $KL$ is a Galois extension over $\QQ$, it is also a Galois extension over $K$ and over $L$. Thus, we may write $\Gal(KL|K)=\{\sigma_1,\ldots,\sigma_m\}$ and $\Gal(KL|L)=\{\tau_1,\ldots,\tau_n\}$. Then
\[\Gal(KL|M)=\{\sigma_i\tau_j:i=1,\ldots,m,j=1,\ldots,n\}\]
and restriction yields isomorphisms $\Gal(KL|K)\to\Gal(L|M)$ and $\Gal(KL|L)\to\Gal(K|M)$.
Let 
\begin{align*}
D=(\tau_i(b_j))_{i,j=1}^n\in\ZZ_K^{n\times n},&&a=(\tau_1(\alpha),\ldots,\tau_n(\alpha))\in\ZZ_M^n,&& b:=\Bigl(\sum_{j=1}^ma_{ij}c_j\Bigr)_{i=1}^n\in L^n.
\end{align*} 
Then 
\[\det(D)^2=\det(D^\text{t}D)=\det((\Tr_{K|M}(b_ib_j)_{i,j}))=d_{K|M}\] 
(here $D^\text{t}$ denotes the transpose of $D$ and $\Tr_{K|M}$ is the trace map of $K$ with respect to $M$). Moreover, $Db=a$. 
Denoting the adjoint matrix of $D$ by $D^*\in\ZZ_K^{n\times n}$ we obtain $\det(D)b=D^*Db=D^*a$. The right hand side is an integral vector and so must be the left hand side. It follows that
\[d_{K|M}a_{ij}=\det(D)^2a_{ij}\in\ZZ_M\subseteq\ZZ_K\]
for all $i,j$. Now by \cite[Corollary~III.2.10]{NeukirchE}, we have 
\[d_K=d_M^{|K:M|}\N_M(d_{K|M})\] 
where $\N_M$ denotes the norm map of $M$ with respect to $\QQ$. Since $M$ is a Galois extension, the norm of $d_{K|M}$ is the product of all Galois conjugates of $d_{K|M}$ in $M$. In particular, $d_{K|M}$ divides $\N_M(d_{K|M})=d_K/d_M^{|K:M|}$ in $\ZZ_M$. Hence, $\frac{d_K}{d_M^{|K:M|}}a_{ij}\in\ZZ_M$ for all $i,j$. By a symmetric argument, $\frac{d_L}{d_M^{|L:M|}}a_{ij}\in\ZZ_M$ and therefore $\frac{\gcd(d_K,d_L)}{d_M^m}a_{ij}\in\ZZ_M$. Hence, we derive
\[\frac{\gcd(d_K,d_L)}{d_M^m}\alpha\in\ZZ_K\ZZ_L\]
as desired.
\end{proof}

It is well-known that $\ZZ_{\QQ_n}=\ZZ[\zeta]$ for every primitive $n$-th root of unity $\zeta$. We also need the following refinements.

\begin{Prop}[Leopoldt, see {\cite[Proposition~6.1]{WangWeiss}}]\label{leopoldt}
Let $K$ be a number field contained in $\QQ_n$. Then $\ZZ_K$ is generated as abelian group by the traces 
\[\sum_{\sigma\in\Gal(K(\zeta)|K)}\sigma(\zeta)\]
of $n$-th roots of unity $\zeta$.
\end{Prop}

\begin{Lem}\label{p2}
Every subfield of $\QQ_{2^n}$ has the form $K=\QQ(\xi)$ where $\xi\in\{\zeta,\zeta\pm\overline{\zeta}\}$ and $\zeta$ is a $2^n$-th root of unity. 
The inclusion of subfields is given as follows
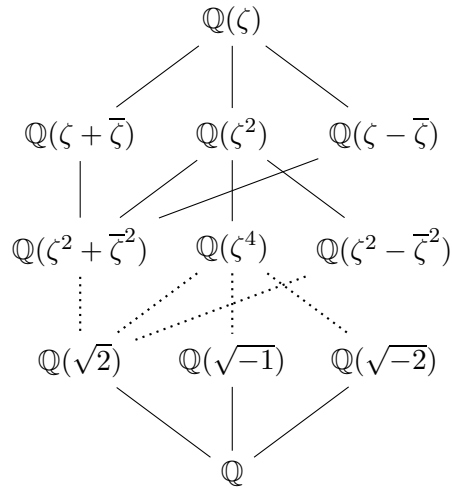
\begin{center}
\begin{tikzpicture}[node distance=15mm]
\draw node (1) {$\QQ(\zeta)$};
\draw node[below of=1] (2) {$\QQ(\zeta^2)$};
\draw node[left of=2,node distance=2cm] (3) {$\QQ(\zeta+\overline{\zeta})$};
\draw node[right of=2,node distance=2cm] (4) {$\QQ(\zeta-\overline{\zeta})$};
\draw node[below of=2] (5) {$\QQ(\zeta^4)$};
\draw node[left of=5,node distance=2cm] (6) {$\QQ(\zeta^2+\overline{\zeta}^2)$};
\draw node[right of=5,node distance=2cm] (7) {$\QQ(\zeta^2-\overline{\zeta}^2)$};
\draw node[below of=5] (8) {$\QQ(\sqrt{-1})$};
\draw node[left of=8,node distance=2cm] (9) {$\QQ(\sqrt{2})$};
\draw node[right of=8,node distance=2cm] (10) {$\QQ(\sqrt{-2})$};
\draw node[below of=8] (11) {$\QQ$};
\draw (1)--(2)--(5);
\draw (1)--(3)--(6);
\draw (1)--(4)--(6);
\draw (2)--(6);
\draw (2)--(7);
\draw[dotted,thick] (5)--(8);
\draw[dotted,thick] (5)--(9);
\draw[dotted,thick] (5)--(10);
\draw[dotted,thick] (6)--(9);
\draw[dotted,thick] (7)--(9);
\draw (8)--(11)--(9);
\draw (10)--(11);
\end{tikzpicture}
\end{center}
If $\xi=\zeta\pm\overline{\zeta}$, then the elements $1$ and $\zeta^k+(\pm\overline{\zeta})^k$ with $k=1,\ldots,2^{n-2}-1$ generate $\ZZ_K$ as abelian group.
\end{Lem}
\begin{proof}
If $n\le 2$, then $K\in\{\QQ,\QQ_4\}$ and the claim holds with $\xi=\zeta\in\{1,\sqrt{-1}\}$. Hence, let $n\ge 3$.
By induction on $n$, we may assume that $K\nsubseteq\QQ_{2^{n-1}}$ and $\zeta$ is a primitive $2^n$-th root of unity.
The subfields of $\QQ_{2^n}$ correspond via Galois theory to the subgroups of the Galois group \[\mathcal{G}:=\Gal(\QQ_{2^n}|\QQ)\cong(\ZZ/2^n\ZZ)^\times\cong C_2\times C_{2^{n-2}}.\]
The involutions of $\mathcal{G}$ are $\alpha:\zeta\mapsto\zeta^{-1}=\overline{\zeta}$, $\beta:\zeta\mapsto\zeta^{-1+2^{n-1}}=-\overline{\zeta}$ and $\gamma:\zeta\mapsto\zeta^{1+2^{n-1}}=-\zeta$. Since $K\nsubseteq\QQ_{2^{n-1}}=\QQ_{2^n}^\gamma$, we must have $\Gal(\QQ_{2^n}|K)\in\{\langle\alpha\rangle,\langle\beta\rangle\}$, i.\,e. $K=\QQ(\zeta\pm\overline{\zeta})$.

As remarked above, $1,\zeta,\ldots,\zeta^{2^{n-1}-1}$ is a $\ZZ$-basis of $\ZZ_{\QQ_{2^n}}$. Hence, every $x\in\ZZ_K$ can be written in the form
\[x=\sum_{k=0}^{2^{n-1}-1}a_k\zeta^k\]
with $a_0,\ldots,a_{2^{n-1}-1}\in\ZZ$. Since $x$ is invariant under $\alpha$ or $\beta$, we obtain $a_k=-(\pm1)^ka_{2^{n-1}-k}$ for $k=1,\ldots,2^{n-1}-1$. Hence,
\[x=a_0+\sum_{k=1}^{2^{n-2}-1}a_k(\zeta^k+(\pm\overline{\zeta})^k)\]
and the second claim follows.
\end{proof}

\begin{Prop}[{\cite[Theorem~3.11]{Navarro2}}]\label{navarro}
Let $G$ be a finite group and $g\in G$. Then the natural map
\[\N_G(\langle g\rangle)/\C_G(g)\to\Gal(\QQ_{|\langle g\rangle|}|\QQ(g))\]
is an isomorphism.
\end{Prop}

\section{General results}

We start our investigation with the “column fields” $\QQ(g)$.
Since products of characters are characters, we have $\ZZ[g]=\sum_{\chi\in\Irr(G)}\ZZ\chi(g)$.

\begin{Prop}\label{Qg}
For every finite group $G$ and $g\in G$ we have
\[\lvert\N_G(\langle g\rangle)/\langle g\rangle\rvert\ZZ_{\QQ(g)}\subseteq\ZZ[g].\]
\end{Prop}
\begin{proof}
Let $n:=|\langle g\rangle|$ and $K:=\QQ(g)\subseteq\QQ_n$. By \autoref{leopoldt}, $\ZZ_K$ is generated by the traces
\[\xi:=\sum_{\sigma\in\Gal(K(\zeta)|K)}\sigma(\zeta)\]
of $n$-th roots of unity $\zeta$. Let $\psi$ be a character of $\langle g\rangle$ such that $\psi(g)=\xi\in K$. Then by \autoref{navarro} it follows that
\[\ZZ[g]\ni(\psi^G)(g)=\frac{1}{|\langle g\rangle|}\sum_{x\in\N_G(\langle g\rangle)}\psi(g^x)=|\N_G(\langle g\rangle)/\langle g\rangle|\xi.\]
This implies $|\N_G(\langle g\rangle)/\langle g\rangle|\ZZ_K\subseteq\ZZ[g]$.
\end{proof}

The following consequence implies Theorem~A.

\begin{Cor}\label{cor}
For every finite group $G$ there exists $e\in\NN$ such that
\[|G|^e\ZZ_{\QQ(G)}\subseteq\ZZ[G].\]
\end{Cor}
\begin{proof}
Clearly, $\QQ(G)=\prod_{g\in G}\QQ(g)$.
By \autoref{disc1}, the discriminants of the fields $\QQ(g)$ for $g\in G$ divide $|G|^{|G|}$. Hence, \autoref{disc2} and \autoref{Qg} imply
\[|G|^e\ZZ_{\QQ(G)}\subseteq |G|^{|G|}\prod_{g\in G}\ZZ_{\QQ(g)}\subseteq\prod_{g\in G}\ZZ[g]\subseteq\ZZ[G]\]
for some (large) $e\in\NN$.
\end{proof}

For specific groups one can estimate the exponent $e$ in \autoref{cor} by using the full strength of Propositions~\ref{disc1} and \ref{disc2}.
For nilpotent groups $G$ we will prove next that $e$ can be taken to be $1$.

\section{Nilpotent groups}

\begin{Lem}\label{dirprod}
Let $G$ and $H$ be finite groups of coprime order. Let $K:=\QQ(G)$ and $L:=\QQ(H)$. Then $\QQ(G\times H)=KL$, $\ZZ_{KL}=\ZZ_K\ZZ_L$ and $\ZZ[G\times H]=\ZZ[G]\ZZ[H]$.
\end{Lem}
\begin{proof}
Since $\Irr(G\times H)=\Irr(G)\times\Irr(H)$, it is clear that $\QQ(G\times H)=KL$ and 
\[\ZZ[G\times H]=\Bigl\{\sum_{i=1}^nx_iy_i:n\in\NN,x_1,\ldots,x_n\in\ZZ[G],y_1,\ldots,y_n\in\ZZ[H]\Bigr\}=\ZZ[G]\ZZ[H].\]
Since $K\subseteq\QQ_{|G|}$ and $L\subseteq\QQ_{|H|}$, the discriminants $d_K$ and $d_L$ are coprime according to \autoref{disc1}.
By \autoref{disc2}, we obtain $\ZZ_{KL}=\ZZ_K\ZZ_L$.
\end{proof}

In the situation of \autoref{dirprod} it is easy to determine $\ZZ_{KL}/\ZZ[G\times H]$ from the elementary divisors of $\ZZ_K/\ZZ[G]$ and $\ZZ_L/\ZZ[H]$. For instance, if 
$\ZZ_K/\ZZ[G]$ has elementary divisors $1,2,4$ (in particular, $\ZZ_K$ has rank $3$) and $\ZZ_L/\ZZ[L]$ has elementary divisors $1,3$, then 
\[\ZZ_{KL}/\ZZ[G\times H]\cong C_2\times C_4\times C_3\times C_6\times C_{12}\cong C_2\times C_6\times C_{12}^2.\]

The following is a special case of Theorem~B.

\begin{Prop}\label{nilpotentodd}
Let $G$ be a nilpotent group of odd order and let $p_1,\ldots,p_n$ be the prime divisors of $|G|$.
Then
\[|G|\ZZ_{\QQ(G)}\subseteq q\ZZ[G]\]
where $q:=\prod_{i=1}^n\min\{p_i^3,|G|_{p_i}\}$.
\end{Prop}
\begin{proof}
We may write $G=P_1\times\ldots\times P_n$ with Sylow subgroups $P_1,\ldots,P_n$.
By \autoref{dirprod}, it follows that
\[|G|\ZZ_{\QQ(G)}=|P_1|\ZZ_{\QQ(P_1)}\ldots|P_n|\ZZ_{\QQ(P_n)}.\]
Thus, we may assume that $G$ is a non-abelian $p$-group for some odd prime $p$. In particular, $|G|\ge p^3$. The Galois group of $\QQ_{|G|}$ (and therefore of every subfield) is cyclic. By \autoref{navarro}, $\Gal(\QQ_{|\langle g\rangle|}|\QQ(g))$ is a cyclic $p$-group for every $g\in G$. 
Hence, the fields $\QQ(g)$ are all cyclotomic and therefore they are totally ordered. In particular, there exists $g\in G$ such that $K:=\QQ(G)=\QQ(g)$. 
By \autoref{Qg}, it follows that $N\ZZ_{K}\subseteq\ZZ[G]$ where $N:=\lvert\N_G(\langle g\rangle)/\langle g\rangle\rvert$.
If $N\le|G|/p^3$, then we are done. So we may assume that $N\ge|G|/p^2$.
If $\QQ(G)=\QQ_p$, then $\ZZ_{K}=\ZZ[\lambda]\subseteq\ZZ[G]$ for any non-trivial linear character $\lambda\in\Irr(G)$. Therefore, we may assume that $|G|\ge p^4$, $|\langle g\rangle|=p^2$ and $\N_G(\langle g\rangle)=\C_G(g)=G$. By \autoref{navarro}, $\QQ(g)=\QQ_{|\langle g\rangle|}=\QQ(\zeta)$ for some root of unity $\zeta$. Since the regular character of $G$ is faithful, there exists $\chi\in\Irr(G)$ such that the restriction $\chi_{\langle g\rangle}$ is faithful. Since $g\in\Z(\chi)$, we have $\chi(g)=\chi(1)\zeta^{k}$ for some integer $k$ coprime to $p$. Then for every $l\ge 0$ we also have $\chi(g^{p^l})=\chi(1)\zeta^{kp^l}$. 
This implies $\chi(1)\ZZ_{K}\subseteq\ZZ[G]$. 
Since $|G|\ge p^4$ and $\chi(1)^2<|G|$, we obtain $|G|\ZZ_{K}\subseteq p^3\ZZ[G]$.
\end{proof}

The analysis of $2$-groups $G$ is more delicate, since it may happen that $\QQ(G)\ne\QQ(g)$ for all $g\in G$. 

\begin{Lem}\label{2fields}
Let $G$ be a $2$-group and $g\in G$ such that $\QQ(g)$ is not a cyclotomic field. Then for every subfield $K$ of $\QQ(g)$ there exists $\chi\in\Irr(G)$ such that $K=\QQ(\chi(g))$.
\end{Lem}
\begin{proof}
We argue by induction on $|G|$. We may assume that $|\QQ(g):\QQ|>2$. In particular, $G\ne 1$.
By \autoref{p2}, the subfields of $\QQ(g)$ are totally ordered. In particular, there exists $\chi\in\Irr(G)$ such that $\QQ(\chi(g))=\QQ(g)$. Let $Z$ be a central subgroup of $G$ of order $2$. Then $\chi^2$ is a character of $G/Z$ and $|\QQ(\chi(g)):\QQ(\chi(g)^2)|\le 2$. Since 
\[\QQ(gZ)=\QQ(\psi(gZ):\psi\in\Irr(G/Z))\subseteq\QQ(g),\]
we obtain $|\QQ(g):\QQ(gZ)|\le 2$. Since $|\QQ(g):\QQ|>2$, also $\QQ(gZ)$ is not a cyclotomic field.
By induction, every proper subfield of $\QQ(g)$ has the form $\QQ(\psi(g))$ for some $\psi\in\Irr(G/Z)$.
\end{proof}

The cyclic group $G=\langle g\rangle\cong C_8$ shows the assumption on $\QQ(g)$ in \autoref{2fields} is necessary.

\begin{Lem}\label{lem2}
Let $G$ be a $2$-group and $g\in G$ such that $K:=\QQ(g)$ is not a cyclotomic field. Then
\[M\ZZ_K\subseteq2\ZZ[G]\]
where $M:=\max\{\chi(1):\chi\in\Irr(G)\}$.
\end{Lem}
\begin{proof}
By \autoref{p2}, there exists a primitive $2^n$-th root of unity $\zeta$ such that $K=\QQ(\zeta\pm\overline{\zeta})$. Moreover, $\ZZ_K$ is generated by the elements $1$ and $\xi_k:=\zeta^k+(\pm\overline{\zeta})^k$ with $k=1,\ldots,2^{n-2}-1$. For every such $k$ there exists $\chi\in\Irr(G)$ such that $\QQ(\chi(g))=\QQ(\xi_k)$ by \autoref{2fields}. It suffices to show that $\chi(1)\xi_k$ is an integral linear combination of the Galois conjugates of $2\chi(g)$. To this end, we may assume that $k=1$ and $\xi:=\xi_1$. 

Let $d:=\chi(1)$ and note that $d>1$ since $\QQ(\chi(g))=\QQ(\xi)=K$ is not a cyclotomic field.
There exist integers $a_0,\ldots,a_{2^{n-1}-1}$ such that
\[\chi(g)=\sum_{i=0}^{2^{n-1}-1}a_i\zeta^i=a_0+\sum_{i=1}^{2^{n-2}-1}a_i\xi_i.\]
Since $\chi(g)$ is a sum of $d$ roots of unity, $|a_0|+\ldots+|a_{2^{n-1}-1}|\le d$ (it may happen that other roots, even of higher order than $2^n$, cancel each other out). The Galois group $\mathcal{G}$ of $\QQ_{2^n}$ acts on $K$ and on $\{\psi(g):\psi\in\Irr(G)\}$.  
Let $\sigma\in\mathcal{G}$ such that $\sigma(\zeta)=\zeta^{1+2^{n-1}}=-\zeta$. Then 
\[\omega:=\sum_{i=0}^{s-1}b_i\xi_{2i+1}=\chi(g)-\sigma(\chi(g))\in\ZZ[G]\]
where $s:=2^{n-3}$ and $b_i:=2a_{2i+1}$ for $i=0,\ldots,s-1$. 
Let $\tau\in\mathcal{G}$ such that $\tau(\zeta)=\zeta^5$. Note that $\tau^s(\xi)=\sigma(\xi)=-\xi$. We may relabel the elements $b_i$ in a suitable order such that
\[\omega=\sum_{i=0}^{s-1}b_i\tau^i(\xi).\]
Next we consider 
\[\gamma:=\sum_{i=0}^{s-1}b_i\zeta^{4i}\in\ZZ_{\QQ_{2s}}.\]
It is known that the prime $2$ is fully ramified in $\QQ_{2s}$. More precisely, $(2)=(\zeta^4-1)^s$ and $(\zeta^4-1)$ is a prime ideal (see \cite[Lemma~I.10.1]{NeukirchE}). Let $e$ be the $2$-part of $\gcd(b_0,\ldots,b_{s-1})$. Then $\frac{1}{e}\gamma$ is an algebraic integer, but $\frac{1}{2e}\gamma$ is not. Hence, 
$(\frac{1}{e}\gamma)=(\zeta^4-1)^t\mathfrak{p}$ where $t<s$ and $\mathfrak{p}$ is an ideal of $\ZZ_{\QQ_{2s}}$ coprime to $(\zeta^4-1)$. This implies the existence of some $\delta\in\ZZ_{\QQ_{2s}}$ such that $\gamma\delta=2em$ where $m$ is an odd integer. We write $\delta=\sum_{i=0}^{s-1}c_i\zeta^{4i}$ with $c_0,\ldots,c_{s-1}\in\ZZ$. 
Then
\[2em=\gamma\delta=\sum_{i,j=0}^{s-1}b_ic_j\zeta^{4(i+j)}.\]
Comparing coefficients yields
\[\sum_{i+j=t}b_ic_j-\sum_{i+j=s-t}b_ic_j=\begin{cases}
2em&\text{if }t=0,\\
0&\text{if }1\le t\le s-1.
\end{cases}\]
Finally we compute
\begin{align*}
\sum_{j=0}^{s-1}c_j\tau^j(\omega)=\sum_{i,j=0}^{s-1}b_ic_j\tau^{i+j}(\xi)=\sum_{t=0}^{s-1}\Bigl(\sum_{i+j=t}b_ic_j-\sum_{i+j=s-t}b_ic_j\Bigr)\tau^t(\xi)=2em\xi.
\end{align*}
Hence, $2em\xi\in\ZZ[G]$. By \autoref{Qg} we also have $|G|\xi\in|G|\ZZ_{\QQ(g)}\subseteq\ZZ[G]$. Therefore, \[2e\xi=\gcd(2em,|G|)\xi\in\ZZ[G].\]
Note that
\begin{equation}\label{eq1}
e\le\sum_{i=0}^{s-1}|b_i|=\sum_{i=0}^{2^{n-2}-1}|a_{2i+1}|\le\sum_{i=0}^{2^{n-1}-1}|a_i|\le d.
\end{equation}
Suppose that $d\xi\notin 2\ZZ[G]$. Then $d\le 2e$ (keep in mind that $d$ and $e$ are $2$-powers). 
If the first inequality in \eqref{eq1} is strict, then $2e\le \sum_{i=0}^{s-1}|b_i|$ since the right hand side is divisible by $e$. Thus, in any case one of the inequalities in \eqref{eq1} is an equality. 
If $e=\sum_{i=0}^{s-1}|b_i|$, then $e=|b_i|$ and $\omega=b_i\tau^i(\xi)$ for some $i\in\{0,\ldots,s-1\}$. Then we obtain $e\xi\in\ZZ[G]$.
If, on the other hand, $\sum_{i=0}^{2^{n-2}-1}|a_{2i+1}|=\sum_{i=0}^{2^{n-1}-1}|a_i|$, then $\omega=2\chi(g)$ and $e\xi\in\ZZ[G]$ by the computation above. Hence in any case we deduce that $d=e$. But now $\chi(g)=a_{2i+1}\tau^i(\xi)$ and $d=2|a_{2i+1}|$. This implies $d\xi\in 2\ZZ[G]$ as desired. 
\end{proof}

The next result is a restatement of Theorem~B. 

\begin{Thm}
For every nilpotent group $G\ne 1$ the exponent of $\ZZ_{\QQ(G)}/\ZZ[G]$ is a proper divisor of $|G|$. 
\end{Thm}
\begin{proof}
By \autoref{nilpotentodd} and its proof, we may assume that $G$ is a $2$-group.
By \autoref{p2}, $\QQ(G)=\QQ(\xi)$ where $\xi\in\{\zeta,\zeta\pm\overline{\zeta}\}$ and $\zeta$ is a primitive $2^n$-th root of unity. 
If there exists $g\in G$ such that $\QQ(G)=\QQ(g)$, then we obtain $|G|\ZZ_{\QQ(G)}\subseteq\ZZ[G]$ from \autoref{Qg}.
Otherwise we have $n\ge 3$, $\QQ(G)=\QQ(\zeta)$ and there exists $g\in G$ such that $K:=\QQ(g)=\QQ(\zeta\pm\overline{\zeta})$. Moreover, there exist $h\in G$ and $\psi\in\Irr(G)$ such that 
\[\psi(h)=\sum_{i=0}^{2^{n-1}-1}a_i\zeta^i\notin K\]
where $a_0,\ldots,a_{2^{n-1}-1}\in\ZZ$.
\autoref{lem2} shows that $M\ZZ_K\subseteq2\ZZ[G]$ where $M:=\max\{\chi(1):\chi\in\Irr(G)\}$. It suffices to prove $|G|\zeta^k\in2\ZZ[G]$ for every $k\in\ZZ$.

Let $\sigma$ be the Galois automorphism of $\QQ(\zeta)$ such that $\sigma(\zeta)=\pm\overline{\zeta}$. 
Since $\psi(h)\notin K$, we have $\psi(h)\ne\sigma(\psi(h))$. 
We consider
\[\omega:=\psi(h)-\sigma(\psi(h))=\sum_{i=1}^{2^{n-1}-1}b_i\zeta^i\in\ZZ[G]\]
where $b_i:=a_i\pm a_{2^{n-1}-i}$ if $i$ is odd and $b_i:=a_i+a_{2^{n-1}-i}$ otherwise. Let $e$ be the $2$-part of $\gcd(b_0,\ldots,b_{2^{n-1}-1})$. As in the proof of \autoref{lem2} there exists an odd integer $m$ such that $2em\omega^{-1}$ is an algebraic integer. Hence for every $k\in\ZZ$,
\[2em\frac{\zeta^k-\sigma(\zeta)^k}{\omega}\in\ZZ_{\QQ(\zeta)}\cap\QQ(\zeta)^\sigma=\ZZ_K.\]
We conclude that
\[2emM\zeta^k=emM(\zeta^k+\sigma(\zeta)^k)+emM\frac{\zeta^k-\sigma(\zeta)^k}{\omega}\omega\in\ZZ[G].\]
By \autoref{cor}, there exists $s\in\NN$ such that $|G|^s\zeta^k\in\ZZ[G]$. Hence,
\[2eM\ZZ_{\QQ(G)}\subseteq\gcd(2emM,|G|^s)\ZZ[\zeta]\subseteq\ZZ[G].\]
If $b_i\ne 0$ for some $i\ne 2^{n-2}$, then $e\le |b_i|\le |a_i|+|a_{2^{n-1}-1}|\le\psi(1)$. Otherwise, $\omega=2a_{2^{n-2}}\sqrt{-1}$. If, in this case, there exists some $a_i\ne 0$ with $i\ne2^{n-2}$, then $e\le|b_{2^{n-2}}|<2|a_{2^{n-2}}|+|a_i|\le 2\psi(1)$. Since $e$ and $\psi(1)$ are $2$-powers, we still have $e\le\psi(1)$. Finally, let $\psi(h)=a_{2^{n-2}}\sqrt{-1}=\omega/2$. Then we may repeat the calculation above with $\psi(h)$ instead of $\omega$ in order to obtain $eM\ZZ_{\QQ(G)}\subseteq\ZZ[G]$ where $e\le2\psi(1)$. In summary, 
\[2M\psi(1)\ZZ_{\QQ(G)}\subseteq\ZZ[G]\]
in every case. Since $|G|=\sum_{\chi\in\Irr(G)}\chi(1)^2$, we have $2M\psi(1)\le 2M^2\le|G|$. If $2M\psi(1)=|G|$, then $\psi(1)=M$ and $\psi$ is the only irreducible character of degree $M$. But then $\psi$ is rational and we derive the contradiction $\psi(h)\in K$. Therefore, $2M\psi(1)<|G|$ and the claim follows.
\end{proof}

\section{Examples}

We show first that \autoref{nilpotentodd} is sharp in the following sense.

\begin{Prop}\label{extra}
For every prime $p$ and every integer $n\ge 1$ there exists a group $P$ of order $p^{2n+2}$ and exponent $p^2$ such that $K:=\QQ(P)=\QQ_{p^2}$ and
$\ZZ_K/\ZZ[P]\cong C_{p^n}^{(p-1)^2}$.
\end{Prop}
\begin{proof}
Let $P$ be the central product of an extraspecial group $E$ of order $p^{2n+1}$ (it does not matter which one) and a cyclic group $C=\langle c\rangle$ of order $p^2$. The irreducible characters of $P$ are those of $E\times C$ which agree on $\Z(E)=\langle z\rangle$ and $\langle c^p\rangle$. It is well-known that $\Irr(E)$ consists of $p^{2n}$ linear character and $p-1$ faithful characters $\chi_1,\ldots,\chi_{p-1}$ of degree $p^n$ (see \cite[Example~7.6(b)]{HuppertChar} for instance). Since $E/E'$ is elementary abelian, the linear character values of $E$ and also of $P$ generate $\QQ_p$. 
Let $\zeta$ be a primitive $p^2$-th root of unity. After relabeling, we may assume that $\chi_i(z)=p^n\zeta^{ip}$ and $\chi_i(g)=0$ for $g\in E\setminus\Z(E)$ and $i=1,\ldots,p-1$. Hence, the non-linear character of $P$ take the values $0$ and $p^n\zeta^i$ for $i\in\ZZ$. This shows $K=\QQ_{p^2}$ and 
\[\ZZ[G]=\ZZ[\zeta^p,p^n\zeta^i:\gcd(i,p)=1].\] 
Since the elements $1,\zeta,\zeta^2,\ldots,\zeta^{p(p-1)-1}$ form a $\ZZ$-basis of $\ZZ_K$, the claim follows easily.
\end{proof}

\autoref{extra} already shows that neither $|\langle g\rangle|\ZZ_{\QQ(g)}\subseteq\ZZ[G]$ nor
$\exp(G)\ZZ_{\QQ(G)}\subseteq\ZZ[G]$ is true in general. Also the dual statements, motivated by \autoref{lem2}, $\chi(1)\ZZ_{\QQ(\chi)}\subseteq\ZZ[G]$ and 
\[\lcm\{\chi(1):\chi\in\Irr(G)\}\ZZ_{\QQ(G)}\subseteq\ZZ[G]\] 
do not always hold. 
Using GAP~\cite{GAP48} and MAGMA~\cite{Magma} we computed the following example: 
The group 
\[G=\texttt{SmallGroup}(48,3)\cong C_4^2\rtimes C_3\] 
gives $K:=\QQ(G)=\QQ_{12}$ and $\ZZ[G]=\ZZ[2\sqrt{-1},\zeta]$ where $\zeta$ is a primitive third root of unity. Hence, $\ZZ_K/\ZZ[G]\cong C_2^2$, but $\lcm\{\chi(1):\chi\in\Irr(G)\}=3$. 

For a single entry $\omega=\chi(g)$ of the character table of $G$ the group $\ZZ_{\QQ(\omega)}/\ZZ[\omega]$ usually has nothing to do with $G$. 
For instance, $G=D_{26}\times C_3$ has a character value $\omega$ such that $\ZZ_{\QQ(\omega)}/\ZZ[\omega]$ is cyclic of order $5^2\cdot157\cdot547$. It is not hard to show that every algebraic integer of an abelian number field occurs in the character table of some finite group (see proof of \cite[Theorem~6]{FeinGordon}).

For $2$-groups the gap between $G$ and $\ZZ_K/\ZZ[G]$ can get even bigger than in \autoref{extra}:
The exponent and the largest character degree of $G=\texttt{SmallGroup}(2^9,6480850)$ is $8$, but
\[\ZZ_K/\ZZ[G]\cong C_{64}\times C_8\times C_4.\]
Similarly, the group $G=\texttt{SmallGroup}(2^9,60860)$ yields $|\ZZ_K/\ZZ[G]|=2^{33}$.

For non-nilpotent groups, the arguments from the last section drastically fail as our next example shows. Let 
\[G=\texttt{SmallGroup}(240,13)\cong C_{15}\rtimes D_{16}\]
where the dihedral group $D_{16}$ acts with kernel $D_{16}'$ (commutator subgroup) on $C_{15}$.
Then $K=\QQ_{120}$ and $2\ZZ_{\QQ(g)}\subseteq\ZZ[G]$ for all $g\in G$, but 
\[ \ZZ_K/\ZZ[G]\cong C_{120}^2\times C_{60}^2\times C_{12}^4\times C_4^4\times C_2^{14}.\]

Now we consider some simple groups which support Conjecture~C.

\begin{Prop}\hfill
\begin{enumerate}[(i)]
\item Let $G=\PSL(2,q)$ for some prime power $q\ne 1$. Then $\ZZ_{\QQ(G)}=\ZZ[G]$.
\item Let $G = \operatorname{Sz}(q)$ for $q \geq 8$ an odd power of $2$. Then $\mathbb{Z}_{\mathbb{Q}(G)}/ \mathbb{Z}[G] \cong C_2^a$ where $a = \varphi((q^2+1)(q-1))/32$.
\end{enumerate}
\end{Prop}
\goodbreak
\begin{proof}\hfill
\begin{enumerate}[(i)]
\item Assume first that $q \geq 5$ is odd. Then $G$ has two irreducible characters taking only rational values and three families $\chi_i, \theta_j, \eta_k$ taking (potentially) irrational values (see \cite[Theorem~38.1]{Dornhoff} for instance). 
Let $\zeta_n$ be a primitive $n$-th root of unity and let $\epsilon:=(-1)^{(q-1)/2}$. Set $r:=(q-1)/2$ and $s:=(q+1)/2$. Then the values of the $\chi_i$ lie in $K:=\QQ(\zeta_r+\overline{\zeta_r})$ and they contain the integral basis from \autoref{p2}. 
Similarly the values of the $\theta_j$ generate the ring of integers of $L:=\QQ(\zeta_s+\overline{\zeta_s})$. Finally, the values of the $\eta_k$ generate the ring of integers of $M:=\QQ(\sqrt{\epsilon q})$. 
The discriminants of  $K$, $L$ and $M$ are pairwise coprime by \autoref{disc1}. Hence, by \autoref{disc2} we have 
\begin{align*} 
\mathbb{Z}[G] &= \mathbb{Z}_K\mathbb{Z}_L\mathbb{Z}_M=\ZZ_{KLM}= \mathbb{Z}_{\mathbb{Q}(G)}.
\end{align*}
For $q$ a power of $2$, the result follows for $\PSL(2,q) = \SL(2,q)$ with a similar argument from \cite[Theorem~38.2]{Dornhoff}.

\item The character table of the group $G = \operatorname{Sz}(q)$ was determined by Suzuki in \cite[Theorem~13]{Suzukidouble}. We use the names of characters in that theorem. Set $r := q-1$, $s:= q + \sqrt{2q} + 1$ and $t:= q - \sqrt{2q} + 1$ and note that these odd numbers are pairwise coprime. Observe that $\mathbb{Q}(G) = KLMN$, the composita of the fields $K = \mathbb{Q}(X_1) = \mathbb{Q}(\zeta_r + \bar{\zeta_r})$, $L = \mathbb{Q}(Y_1) = \mathbb{Q}(\zeta_s + \zeta_s^q + \zeta_s^{q^2} + \zeta_s^{q^3})$, $M = \mathbb{Q}(Z_1) = \mathbb{Q}(\zeta_t + \zeta_t^q + \zeta_t^{q^2} + \zeta_t^{q^3})$ and $N = \mathbb{Q}(W_1) = \mathbb{Q}(\sqrt{-1})$, which have pairwise coprime discriminant by \autoref{disc1}. Now $\mathbb{Z}_{K} = \mathbb{Z}[\zeta_r + \bar{\zeta_r}] = \mathbb{Z}[X_1]$ and $\mathbb{Z}_{L} = \mathbb{Z}[\zeta_s + \zeta_s^q + \zeta_s^{q^2} + \zeta_s^{q^3}] = \mathbb{Z}[Y_1]$ and similarly for $Z_1$. Further $\mathbb{Z}[W_1] = \mathbb{Z}[W_2] = \mathbb{Z}[2\sqrt{-1}]$, hence $\mathbb{Z}_N/\mathbb{Z}[W_1]$ has elementary divisors $1$ and $2$. Similar to the remark following \autoref{dirprod}, we can conclude that $\mathbb{Z}_{KLMN}/\mathbb{Z}[G]$ has elementary divisors $1$ and $2$ each with multiplicity 
\[[KLM:\mathbb{Q}] = \frac{\varphi(r)}{2}\frac{\varphi(s)}{4}\frac{\varphi(t)}{4} = \frac{\varphi((q^2+1)(q-1))}{32}.\qedhere\]
\end{enumerate}
\end{proof}

A minimal simple group (i.\,e. a simple group with all proper subgroups solvable) is isomorphic to some $\operatorname{PSL}(2, q)$, to some $\operatorname{Sz}(2^{2f + 1})$ or to $\operatorname{PSL}(3, 3)$. For the last group one can check easily that $\mathbb{Z}_{\mathbb{Q}(G)} = \mathbb{Z}[G]$. Hence, for minimal simple groups $G$, the exponent of $\mathbb{Z}_{\mathbb{Q}(G)}/ \mathbb{Z}[G]$ is at most $2$. 

Finally we compute $\ZZ[G]$ for the alternating group $G=A_n$ of (small) degree $n$. 
Let $g\in G$ be non-rational. Then there exists a partition $\lambda=(\lambda_1,\ldots,\lambda_k)$ of $n$ into pairwise distinct odd parts such that 
\[\ZZ[g]=\ZZ[(1+\sqrt{d})/2]\] 
where $d=(-1)^{(n-k)/2}\lambda_1\ldots\lambda_k\equiv 1\pmod{4}$ (see \cite[Theorem~2.5.13]{JamesKerber} for instance). We may write $\sqrt{d}=e\sqrt{d'}$ such that $d'$ is squarefree. Let $K:=\QQ(g)=\QQ(\sqrt{d})=\QQ(\sqrt{d'})$. Then 
\[\ZZ_K=\ZZ[(1+\sqrt{d'})/2]\] 
and we obtain $e\ZZ_K\subseteq\ZZ[g]$. Note that $e^2\mid d\mid n!=2|G|$. 
Since the discriminant of $K$ is $d'\equiv 1\pmod{2}$, it follows that $|\ZZ_{\QQ(G)}/\ZZ[G]|$ is odd by \autoref{disc2}. It seems fairly difficult to determine the precise structure of $\ZZ_{\QQ(G)}/\ZZ[G]$. 
For $n\ge 25$, a theorem by Robinson--Thompson~\cite{RobThomp} states that 
\[\QQ(G)=\QQ(\sqrt{p^*}:p\text{ odd prime }, n-2\ne p\le n)\]
where $p^*:=(-1)^{\frac{p-1}{2}}p$. By \autoref{disc2}, $\ZZ_{\QQ(G)}$ is generated as abelian group by all products of the elements $(1+\sqrt{p^*})/2$ with $p$ as above.
The following table lists the (non-trivial) elementary divisors of $\ZZ_{\QQ(G)}/\ZZ[G]$ for $n\le 31$. In every case Conjecture~C is fulfilled.

\[
\begin{array}{ll}
n&\ZZ_{\QQ(A_n)}/\ZZ[A_n]\\\hline
\le 11&1\\
12,13,14&3^4\\
15&3^4\times {15}^4\times {45}^4\\
16&3^4\times {15}^4\\
17&3^{12}\times 9^4\times {45}^4\times {135}^4\\
18&3^8\times {15}^8\times {45}^8\\
19&3^8\times {15}^8\\
20&3^{36}\times 9^{12}\times {45}^{32}\times {10395}^{28}\times {31185}^4\\
21&3^{36}\times {105}^4\times {315}^{12}\\
22&3^{52}\times {105}^8\times {315}^{52}\times {945}^4\\
23&3^{64}\times {4095}^{32}\\
24&1\\
25&3^{32}\times {15}^{32}\times {315}^{32}\\
26&3^{38}\times {15}^{40}\times {45}^{40}\times {315}^{56}\times {945}^8\\
27&3^{112}\times 9^{112}\times 27^{16}\\
28&3^{96}\times 15^{80}\times 45^{48}\\
29&3^{224}\times 15^{128}\\
30&3^{128}\times 105^{128}\\
31&3^{256}
\end{array}
\]

\section*{Acknowledgment}
The work on this paper started with a visit of the first author at the University of Jena in January 2019. He appreciates the hospitality received there. The authors also like to thank Thomas Breuer for making them aware of the CoReLG package~\cite{Corelg} of GAP~\cite{GAP48} which was used for computations with alternating groups. The first author is a postdoctoral researcher of the FWO (Research Foundation Flanders).
The second author is supported by the German Research Foundation (\mbox{SA 2864/1-1} and \mbox{SA 2864/3-1}).

\end{document}